\documentclass[11pt]{amsart}
\usepackage{amssymb}
\usepackage[headings]{fullpage}
\usepackage{amsfonts}
\usepackage{paralist}
\usepackage{tikz}
\usepackage[dvipsnames]{xcolor}
\usepackage{graphicx}
\usepackage{subcaption}
\usepackage[shortlabels]{enumitem}
 
\usepackage[colorlinks,pagebackref=true,pdftex]{hyperref}

\makeatletter
\@addtoreset{equation}{section}
\def\theequation{\thesection.\@arabic \c@equation}
    
\def\@citecolor{blue}
\def\@linkcolor{blue}
\def\@urlcolor{blue}
 
\newtheorem{theorem}[equation]{Theorem}
\newtheorem{lemma}[equation]{Lemma}

\newtheorem{proposition}[equation]{Proposition}
\newtheorem{corollary}[equation]{Corollary}

\newtheorem{claim*}{Claim}

\theoremstyle{definition}
\newtheorem{remark}[equation]{Remark}

\newtheorem{eg}[equation]{Example}
\newenvironment{example}[1][]{\begin{eg}[#1] \pushQED{\qed}}{\popQED
\end{eg}}
\newtheorem{definition}[equation]{Definition}

\newtheorem{notn}[equation]{Notation}

\renewcommand{\k}{\Bbbk}
\newcommand{\mP}{\mathcal P}
\newcommand{\N}{\mathbb N}
\newcommand{\Z}{\mathbb Z}
\newcommand{\la}{\langle}
\newcommand{\ra}{\rangle}

\usepackage{scalerel}
\newcommand{\precdot}{\mathrel{\prec\mathrel{\mkern-8mu}\scalerel*{\cdot}{\prec}}}

\title{KW Semigroups - Their Betti Numbers, Ap\'ery Posets and Tangent Cones}

\author[M. Gonz\'alez-S\'anchez]{Mario Gonz\'alez-S\'anchez} 
\address{Instituto de Investigaci\'on en Matem\'aticas de la Universidad de Valladolid (IMUVA), Universidad de Valladolid, 47011 Valladolid, Spain.}
\email{mario.gonzalez.sanchez@uva.es}
\thanks{The first author is supported in part by the grant PID2022-137283NB-C22 funded by MICIU/AEI/ 10.13039/501100011033 and by ERDF/EU, and thanks financial support from European Social Fund, {\it Programa Operativo de Castilla y Le\'on}, and {\it Consejer\'ia de Educaci\'on de la Junta de Castilla y Le\'on.}
In addition, he thanks the University of Missouri for the financial support received for a two-month visit to MU during Spring 2026.}

\author[H. Srinivasan]{Hema Srinivasan}
\address{Department of Mathematics, University of Missouri, Columbia MO 65211, USA.}
\email{srinivasanh@missouri.edu}
\thanks{Second author is partially supported by Simons Foundation.}


\begin{document}

\begin{abstract}
Let $p<q$ be coprime integers. Kunz–Waldi semigroups are numerical semigroups containing $p$ and $q$ and contained in $\la p,q,r \ra$, where $2r = p,q,p+q$ whichever is even.
In this paper, we prove a conjecture on the Betti numbers of the semigroup rings of these semigroups, showing that they coincide with those of the ideal of $2\times 2$ minors of a $2 \times n$ generic matrix, where $n$ is the embedding dimension.
Moreover, we characterize the Ap\'ery posets of Kunz–Waldi semigroups and determine when their tangent cones are Cohen–Macaulay.
\end{abstract}

\maketitle

\section{Introduction}
A numerical semigroup $H$ containing two relatively prime integers $p < q$ is called a {\it Kunz-Waldi semigroup} (or {\it KW semigroup}) if $\la p,q \ra \subsetneq H \subset  \la p,q,r \ra$ where $2r = p $ or $q$  or $p+q$ whichever is even. These semigroups were introduced and studied by Kunz and Waldi in \cite{KW14}.  Let $KW(p,q)$ denote the class of such semigroups.  
KW semigroups are in one-to-one correspondence to the lattice paths, with right and downward steps, in the rectangle $\mathbf{R} \subset \N^2$ with vertices $(0,0)$, $(0,p'-1)$, $(q'-1,p'-1)$, and $(q'-1,0)$,
where $p' = \lfloor p/2 \rfloor$ and $q' = \lfloor q/2 \rfloor$. In fact, a semigroup $H = \la p,q,h_1,\ldots,h_{n-2} \ra \in KW(p,q)$ if and only if $h_i = pq-x_ip-y_iq$ satisfying $0 < x_1 < \dots < x_{n-2} \leq q/2$ and $p/2 \geq y_1 > \dots > y_{n-2} > 0$, as observed in \cite[Rem. 1]{SS25}, and the embedding dimension of $H$ is $e(H) = n$. The corners of the lattice path corresponding to the semigroup $H$ have coordinates $(x_i-1,y_i-1)$, $1\leq i \leq n-2$.

In \cite{KW17}, Kunz and Waldi determine the minimal generators and type for the defining ideals of these semigroup rings. Indeed, they show that for a KW semigroup of embedding dimension n, the minimal number of generators and the type are $n \choose 2$ and $n-1$, respectively.  
In \cite{GSS25}, the authors give a complete characterization of when the semigroup rings of KW semigroup define determinantal varieties. They also study when KW semigroups lie in the interior of the same face of the Kunz cone as a determinantal KW semigroup, in which case they have determinantal Betti numbers.
We say that a numerical semigroup, or its associated variety, has {\it determinantal Betti numbers} if the Betti numbers are those of a determinantal variety defined by the $2\times 2$ minors of a generic $2\times n$ matrix.    
The subclass of semigroups in $KW(p,q)$ whose variety is determinantal is denoted by $KW_D(p,q)$. By \cite[Thm. 1.4]{GSS25}, a KW semigroup $H = \la p,q,h_1,\ldots,h_{n-2}\ra \in KW_D(p,q)$ if and only if there exist $x,y \in \Z_{>0}$ such that $(n-2)x \leq q/2$, $(n-2)y \leq p/2$, and $h_i = pq-x_ip-y_iq$, where $x_i = ix$ and $y_i = (n-1-i)y$ for all $1\leq i \leq n-2$.

It is conjectured in \cite{GSS25} that indeed all KW semigroups have determinantal Betti numbers and hence the Betti numbers of KW semigroups depend only on their embedding dimension.
In this paper, we settle that conjecture in affirmative. 
This is the main result of Section \ref{sec:bettti}, Theorem \ref{KWconj}.
We prove that the Betti numbers $\beta_i$ of a KW semigroup of embedding dimension $n$ are $\beta_i = i{n\choose i+1}, 1\le i \le n-1$.  This is the same as if they are generated by the $2\times 2$ minors of a $2\times n$ generic matrix.  In particular, this generalizes the result in \cite{KW17} which says the first and the last Betti numbers are $n\choose 2$ and $n-1$, respectively. 

A major ingredient in describing the Kunz Cone and its faces is the Ap\'ery poset. This was the main thrust in one half of the proof of the main theorem.  In Section \ref{sec:Apery}, we give a complete characterization of the Ap\'ery poset of a KW semigroup, this is the content of Theorem \ref{thm:Apery_KW}. Any semigroup with this Ap\'ery poset, will have determinantal Betti numbers, by Kunz's theorem \cite{Kunz87}.  However, it is interesting to note that one can have semigroups rings with determinantal Betti numbers whose  Ap\'ery posets are not like that of a KW semigroup; see Example \ref{ex:Ap_arithm}.

In Section \ref{sec:tangent_cone}, we consider the tangent cone $G(H)$ of KW semigroup rings and give a necessary and sufficient condition for these to be Cohen-Macaulay and prove that $G(H)$ has the same determinanatal Betti numbers when it is Cohen-Macaulay in Theorem \ref{thm:tangentcone_CM}. We also give examples to show that the Betti numbers will not be the same in non Cohen-Macaulay case. 

In this paper, $\k$ will denote an arbitrary field.

\section{Betti Numbers of KW semigroups} \label{sec:bettti}

Let $p<q$ be two relatively prime positive integers. So exactly one of $p,q, p+q$ must be even.  Let $r =p/2 $ or $q/2$ or $(p+q)/2$ whichever is an integer.  A KW semigroup $H$ is one that contains $p,q$ and is contained in $\la p,q,r \ra$.  We say $H\in KW(p,q)$ if it is a semigroup of this form.  
It can be seen that if $H \in KW(p,q)$ of embedding dimension $n$, then there exist two sequences of integers $0<x_1<x_2<\dots<x_{n-2} \leq p/2$ and $q/2 \geq y_1>y_2> \dots > y_{n-2} >0$ such that $H$ is minimally generated by $\{ p,q, h_1, \ldots,h_n\}$ where $h_i = pq-x_ip-y_iq$. 
These are the semigroups introduced and studied in \cite{KW14}.  

In this section, we will prove that all KW semigroups have determinantal Betti numbers.  We will make use of two kinds of results about numerical semigroups having the same Betti numbers.  First is a Kunz's Semigroup Theorem \cite{Kunz87} which says that any two numerical semigroups lying in the interior of the same face of the Kunz Cone have the same Betti numbers. Second is a version of a result of Morales \cite{Mor87,Mor91} which says that the Betti numbers are not changed when we multiply all but one minimal generator by a suitable positive integer.  We will state and prove this in some generality for finitely generated semigroups in $\N^m$ in this section.

Let $a_i = (a_{i1},\ldots,a_{im}) \in \N^m$, $1\leq i \leq n$, minimally generate a semigroup $S = \la a_1,\ldots,a_n\ra$ in $\N^m$. Consider two polynomial rings $R = \k[t_1,\ldots,t_n]$, with the $S$-grading given by $\deg(t_i) = a_i \in \N^m$, and $\k[w_1,\ldots,w_m]$. The defining ideal of $S$ is the $S$-homogeneous binomial prime ideal $I_S = \ker (\varphi_S)$, where $\varphi_S: R \rightarrow \k[w_1,\ldots,w_m]$ is defined by $\varphi_S(t_i) = w_1^{a_{i1}}\dots w_m^{a_{im}}$. Using these notations, the semigroup ring of $S$, $\k[S] = \k[w^s \mid s\in S]$, is isomorphic to $R/I_S$ as graded $R$-modules.

\begin{lemma} \label{lem:lambda}
Let $a_i \in \mathbb N^m$ for $1\le i\le n$ minimally generate a semigroup in $\mathbb N^m$. Let $A = [a_1, \ldots, a_n]$ be the $m\times n$ matrix over $\mathbb N$.  Let $1\leq i \leq n$ such that $\lambda_i = \min \{ b > 0 \mid  ba_i = \sum_{j\neq i}b_ja_j,  b_j\in \mathbb Z\} \in \N$.  If  $AT = 0$ for some $T = [t_1,\ldots,t_n] \in \Z^n$, then $\lambda_i $ divides $t_i$. 
\end{lemma}
\begin{proof}
By definition of $\lambda_i$, there is a $Z = [z_1,\ldots,z_n]^T \in \mathbb Z^n$ with $z_i = \lambda _i$ such that $AZ = 0$. 
Let $t_i = \lambda_i q+r$, $0\le r< \lambda_i$.  Then 
$A(T- qZ) = 0$.  But then $ra_i$ is in the group generated by the rest of the $a_j$ in $\mathbb Z^m$.  So, $r= 0$. 
\end{proof}

 \begin{theorem} \label{thm:Morales_Nm}
Let $a_i \in \mathbb N^m$ for $1\le i\le n$ minimally generate a semigroup in $\mathbb N^m$. Let $\lambda_i = \min \{ b > 0 \mid ba_i = \sum_{j\neq i}b_ja_j,  b_j\in \mathbb Z\}$.  If $\lambda_i \in \N$, then the Betti numbers of the semigroup $S= \la a_1, \ldots a_n \ra$ are the same as those of $S_i= \la a_1 \ldots, \lambda_i a_i, \ldots a_n \ra, 1\le i\le n$. 
\end{theorem}

\begin{remark}
Note that if $\lambda_i = \infty$, i.e. $a_i$ does not belong to the group generated by $\{a_j \mid j\neq i\}$, then the defining ideals of $S$ and $S_i$ coincide, so $S$ and $S_i$ have the same Betti numbers.
\end{remark}

\begin{proof}
Denote by $I_S \subset \k[t_1,\ldots,t_n]$ the defining ideal of $S$ and $I_{S_i} \subset \k[z_1,\ldots,z_n]$ the defining ideal of $S_i$.
By Lemma \ref{lem:lambda}, if $f(t_1, \ldots, t_n)$ is a binomial in $I_S$, then $f = g(t_1,\ldots,t_i^{\lambda_i},\ldots,t_n)$ for some binomial $g$.  Let $F\to R/I_{S_i}$ be the minimal resolution with maps $\delta_k: F_k \to F_{k-1}$.   By replacing the $z_i$'s in the entries of the matrices of $\delta_k$ with $z_i^{\lambda_i}$ we get a resolution $F' \to R/I_S$.  
That is, if we map $\phi(t_i)= z_i^{\lambda_i}$ and $\phi(t_j) =z_j, j\neq i$,  we get an isomorphism  $$\phi: \k[t_1,\ldots,t_n]/I_S \cong   \k[z_1,\ldots, z_i^{\lambda_i}, \ldots, z_n]/I_{S_i}.$$  
This  induces  an isomorphism between the minimal resolutions.   So, the Betti numbers are equal. 
\end{proof}

\begin{theorem} [Morales Theorem] \label{thm:Morales}
Let $a_i \in \mathbb N$ for $1\le i\le n$ minimally generate a semigroup in $\mathbb N$. 
Then, the Betti numbers of the semigroup $S=\la a_1, \ldots a_n \ra$ are the same as those of $S_k= \la a_1, ka_2, \ldots, ka_n \ra$ as long as $\gcd(a_1, k) = 1$.  
\end{theorem}

\begin{proof}
This is a direct consequence of Theorem \ref{thm:Morales_Nm}. Let $\lambda_1 = \min \{ b>0 \mid ba_1 = \sum_{j=2}^n b_j k a_j, b_j\in \Z\}$. Since $\gcd(a_1,k) = 1$, then $\lambda_1 = k$. Thus, by Theorem \ref{thm:Morales_Nm}, the Betti numbers of $S_k$ are the same as the Betti numbers of $\la ka_1,\ldots,ka_n\ra \cong  \la a_1,\ldots,a_n\ra$ and hence we get the result.
\end{proof}

Note that in this theorem, $a_1$ is any minimal generator of $S$. As a direct consequence, we get the following result.

\begin{corollary}
Under the hypothesis of Theorem \ref{thm:Morales}, let $I_S$ be the defining ideal of $S = \la a_1,\ldots,a_n\ra$ and $I_{S_k}$ be the defining ideal of $S_k = \la a_1,ka_2,\ldots,ka_n\ra$.
If $F \rightarrow R/I_S$ is the minimal resolution of $I_S$ and we set $F(k) = F \otimes_R \k[t_1^k,t_2,\ldots,t_n]$, then $F(k) \rightarrow R/I_{S_k}$ is the minimal resolution of $R/I_{S_k}$.

Thus, if $S= \la a_1, \ldots, a_n \ra $ is complete intersection, Gorenstein or determinantal so are the semigroups $S_k =  \la a_1,ka_2,\ldots,ka_n \ra$ as long as $\gcd (k,a_1) =1$.

If $F$ is the Frobenius number of $S$, $kF+a_1$ is the Frobenius number of $S_k$.
\end{corollary}

\begin{remark}
This is in contrast to the Gluing.  Indeed $S_k$ is not a gluing of $a_1$ with the rest because $a_1 \notin \la a_2, \ldots, a_n \ra$.  
If it was the gluing, then its Frobenius number would be $kF+a_1(k-1)$.
Moreover, while Gluing preserves the complete intersection and symmetric properties, it does not preserve the Betti numbers. In fact, the Betti numbers of $S_k = \la a_1,ka_2,\ldots,ka_n\ra$ are given by 
\[\beta_i (S_k) = \begin{cases}
\beta_i(S) & \text{if }   a_1\not\in \la a_2, \ldots, a_n \ra, \\
\beta_i (S)+\beta_{i-1}(S) & \text{if } a_1 \in \la a_2, \ldots, a_n \ra ,
\end{cases} \]
where the equality in the first case follows from Theorem \ref{thm:Morales}, and the equlity in the second case follows from \cite[Cor. 3.2]{GS19}.
\end{remark}

Recall that for a KW semigroup $H = \la p,q,h_1,\ldots,h_{n-2} \ra \in KW(p,q)$, with $h_i = pq-x_ip-y_iq$ for $1\leq i \leq n-2$, one has that $I_H$ is determinantal if and only if there exist positive integers $x,y$ such that $x_i=ix$ and $y_i = (n-1-i)y$ for all $1\leq i \leq n-2$. In \cite[Cor. 2.8]{GSS25}, the authors show that for any $H \in KW(p,q)$ with the same sequence of $y_i$'s as a determinantal KW semigroup, $H$ has determinantal Betti numbers. In the next result, we  show that the conclusion holds when $H$ has the same sequence of $x_i$'s as a determinantal KW semigroup.
 
\begin{proposition}\label{prop:symmetry} 
Let $H = \la p,q, h_1, \ldots, h_{n-2} \ra \in KW(p,q)$, $h_i = pq-x_ip-y_iq$ for $1\leq i \leq n-2$, be the semigroup defined by the sequences $0<x_1<x_2<\dots<x_{n-2} \leq q/2$ and $p/2 \geq y_1>y_2>\dots>y_{n-2}>0$.
\begin{enumerate}[(1)]
    \item\label{prop:symmetry_1} If there exists $x\in \Z_{>0}$ such that $(n-2)x \leq q/2$ and $x_i = ix$ for all $1\leq i \leq n-2$, then $H$ has determinantal Betti numbers, i.e. $\beta_i(\k[H]) = i {n \choose i+1}$, $1\leq i \leq n-1$.
    
    \item\label{prop:symmetry_2} If there exists $y\in \Z_{>0}$ such that $(n-2)y \leq p/2$ and $y_i = (n-2-i)y$ for all $1\leq i \leq n-2$, then $H$ has determinantal Betti numbers, i.e. $\beta_i(\k[H]) = i {n \choose i+1}$, $1\leq i \leq n-1$.
    
\end{enumerate}  
\end{proposition}

\begin{proof}
Part \ref{prop:symmetry_2} is \cite[Cor. 2.8]{GSS25}. Let us prove \ref{prop:symmetry_1}.

Take $k$ such that $\gcd(k,p) = 1$ and $q<kp$. Then, $kh_i = q(kp) - q (ky_i) - (kp) x_i$, with $0 < ky_{n-2} < ky_{n-1} < \dots < ky_1 \leq kp/2$ and $q/2 \geq x_{n-2} > x_{n-1} > \dots > x_1 > 0$.
So, $H_k = \la q, kp, kh_{n-2},\ldots,kh_1 \ra \in KW(q,kp)$. 

Note that by hypothesis, $h_{n-1-i} = q(kp)-q (ky_{n-1-i})-kp(n-2-i)x$, so $H_k$ has determinantal Betti numbers by part \ref{prop:symmetry_2}. To conclude the result, note that the Betti numbers of $H$ are the same as those of $H_k$, by Theorem \ref{thm:Morales}.
\end{proof}

\begin{theorem}\label{KWconj}
The Betti numbers of every $KW$ semigroup $H$ of embedding dimension $n$ are 
\[\beta_i(\k[H]) = i{n\choose i+1}, \ 1\leq i \leq n-1.\]
\end{theorem}
\begin{proof}

Let $H$ be a $KW$ semigroup of embedding dimension $n$.    So, $H = \la p,q,h_1,\ldots,h_{n-2} \ra \in KW(p,q)$ for some relatively prime positive integers $p<q$.  Further, there exist positive integers $x_1 < \ldots <x_{n-2} \le q/2$ and $y_{n-2} < \ldots y_1 \le p/2$ such that $h_i = pq-x_ip-y_iq$, $1\leq i \leq n-2$. 

Let $H' = \la p,q, h'_1,\ldots,h_{n-2}' \ra$, where $h'_i= pq-ip - y_i q$, $1\leq i \leq n-2$. Then $H' \in KW(p,q)$ and the semigroups $H$ and $H'$ are defined by the same sequence of $y_i$'s. Hence, they belong to the interior of the same face of the Kunz cone for the semigroups with multiplicity $p$, by \cite[Prop. 2.7]{GSS25}. Therefore, by Kunz's Semigroup theorem \cite{Kunz87}, $H$ and $H'$ have the same Betti numbers. But $H'$ has determinantal Betti numbers, by part \ref{prop:symmetry_1} of Proposition \ref{prop:symmetry}. So, $H$ has determinantal Betti Numbers. 
\end{proof}

\begin{remark}
Above Theorem \ref{KWconj} shows that the Betti numbers of a KW semigroup are determinantal and depend only on the embedding dimension. In particular, they are independent of the multiplicity.  For example the semigroup $H_k$ in the proof of Proposition \ref{prop:symmetry} has multiplicity $q$ and $H$ has multiplicity $p$.  So they are not even in the same Kunz cone!

\end{remark}

This raises the question of whether the Betti numbers are the same for all the semigroups of the form $\la p,q,h_1,\ldots,h_{n-2} \ra $ with no restrictions on the $h_i$'s. 

As Ap\'ery posets determine which semigroups are in the interior of the same face of the Kunz Cone, in the next section we analyze the Ap\'ery posets of KW semigroups. 

\section{Ap\'ery poset of KW semigroups} \label{sec:Apery}

Let $S$ be a numerical semigroup with multiplicity $m = \min (S\setminus \{0\})$. The {\it Ap\'ery set} of $S$ is ${\rm Ap}(S) = \{b \in S \mid b-m \notin S\}$, and one can write ${\rm Ap}(S) = \{b_0:=0,b_1,...,b_{m-1}\}$, where $b_i$ is the least element of $S$ congruent to $i$ modulo $m$.

\begin{definition}
The \textit{Ap\'ery poset} of $S$ is $\mathcal{P}(S) = (\Z_m , \preceq)$, where $i \preceq j$ if and only if $b_j-b_i \in S$ for $i,j \in \Z_m $. 
We write $i \precdot j$ and say $j$ {\it covers} $i$ if $i \prec j$ and there is no $k$ such that $i \prec k \prec j$. 
\end{definition}

The aim of this section is to provide a complete characterization of the Ap\'ery poset of Kunz-Waldi semigroups.
For convenience, first we recall here what the Ap\'ery set and the Ap\'ery poset of a Kunz-Waldi semigroup $H$ look like.

\begin{proposition}[{\cite[Prop. 2.1 and 2.6]{GSS25}}] \label{prop:Apery_KW}
Let $H = \langle p,q,h_1,\ldots,h_{n-2} \rangle \in KW(p,q)$ be the semigroup defined by the sequences $x_1<x_2<\dots < x_{n-2}$ and $y_1>y_2>\dots > y_{n-2}$, and set $y_{n-1} := 0$. The Ap\'ery set of $H$ is
\[{\rm Ap}(H) = \{\lambda q \mid 0\leq \lambda < p-y_1 \} \cup \left( \cup_{i=1}^{n-2} \{h_i +\lambda q \mid 0\leq \lambda < y_i-y_{i+1}\} \right) \, .\]

For $0 \leq \lambda < p-y_1$ set $i_{0,\lambda}$ the label of $\lambda q$ in $\mathcal{P}(H)$, i.e. $0\leq i_{0,\lambda} < p$ and $i_{0,\lambda} \equiv \lambda q \pmod{p}$. 
Similarly, for each $1\leq j \leq n-2$ and $0\leq \lambda < y_j-y_{j+1}$, let $i_{j,\lambda}$ be the label of $h_j+\lambda q$ in $\mathcal{P}(H)$.
Then, the covering relations in ${\mathcal P}(H)$ are given by 
\[i_{j_1,k_1} \precdot i_{j_2,k_2} \Leftrightarrow \left\{ 
\begin{array}{lcl}
 j_2 = j_1 & \text{and} & k_2=k_1+1 \text{, or} \\
 j_1=0 & \text{and} & k_2 = k_1 \, .
\end{array}\right.\]
\end{proposition}

\begin{example} \label{ex:Apery}
Consider the semigroup $H = \la 8,17,60,69,78\ra \in KW(8,17)$, defined by the sequences $x_1=1<x_2=2<x_3$ and $y_1=4>y_2=3>y_3=2$. Its Ap\'ery set is ${\rm Ap}(H) = \{0, 17,34,51,60,69,78,95\}$ and the Hasse diagram of ${\mathcal P}(H)$ is shown in Figure \ref{fig:example_Hasse}.
\end{example}

\begin{figure}[htbp]
\centering
\begin{tikzpicture}
    \node (0) at (0,0) {$0$};
    \node (17) at (-1,1) {$1$};
    \node (34) at (-1,2) {$2$};
    \node (51) at (-1,3) {$3$};
    \node (60) at (0,1) {$4$};
    \node (69) at (1,1) {$5$};
    \node (78) at (2,1) {$6$};
    \node (95) at (1.5,2) {$7$};
    \draw (0) -- (17) -- (34) -- (51);
    \draw (0) -- (60);
    \draw (0) -- (69);
    \draw (0) -- (78);
    \draw (78) -- (95);
    \draw (95) -- (17);
\end{tikzpicture}
\caption{Hasse diagram of the poset $\mP(H)$ in Example \ref{ex:Apery}.}
\label{fig:example_Hasse}
\end{figure}
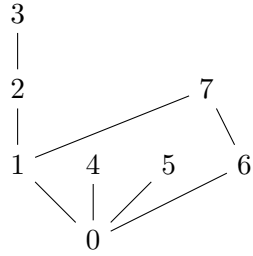

The following theorem characterizes the Ap\'ery poset of Kunz-Waldi semigroups.

\begin{theorem} \label{thm:Apery_KW}
Let $p\geq 3$ be a prime number and $\mP = \left( \Z_p, \preceq \right)$ be a poset on 
\[ \Z_p = \{\lambda a \mid 0 \leq \lambda \leq \ell_0\} \cup \left( \cup_{i=1}^{n-2} \{k_i +\lambda a \mid 0 \leq \lambda \leq \ell_i \}\right) \]
for some $a\in \Z_p$ with $\gcd(a,p) = 1$, $\ell_i \in \N$, $\ell_0\geq 1$ and $k_i \in \Z_p$ (all the $k_i$'s different) such that $p = ( \sum_{i=0}^{n-2} \ell_i ) + n-1$. Then, $\mP$ is the Ap\'ery poset of a semigroup $H \in KW(p,q)$ for some $q>p$ coprime with $p$ if and only if:
\begin{enumerate}[(1)]
    \item\label{thm:Apery_KW_1} $\ell_i < \ell_0$ for all $1\leq i \leq n-2$, and the covering relations in $\mP$ are exactly the following
    \[\left\{ \begin{array}{ll}
      ja \precdot (j+1)a,   & 0 \leq j \leq \ell_0-1; \\
      k_i+j a \precdot k_i + (j+1)a,   & 1\leq i \leq n-2, \, 0 \leq j \leq \ell_i-1; \\
      j a \precdot k_i + ja,   & 1\leq i \leq n-2, \, 0 \leq j \leq \ell_i-1; \\
    \end{array} \right.\]
    i.e., the Hasse diagram of $\mP$ is as shown in Figure~\ref{fig:Hasse}.
    Hence, the poset $\mP$ is graded.
    
    \item\label{thm:Apery_KW_2} $y_1 \leq \lfloor \frac{p}{2} \rfloor$, where $y_1 := \left( \sum_{i=1}^{n-2} \ell_i \right) + n-2 = p-1-\ell_0$.
    
    \item\label{thm:Apery_KW_3} There exist natural numbers $0< y_{i_1} ,y_{i_2},\ldots , y_{i_{n-2}} \leq y_1$ such that \[-a y_{i_j} \equiv k_j \pmod{p} , \quad 1\leq j \leq n-2,\]
    and there exists $1\leq j \leq n-2$ such that $y_{i_j} = y_1$.\\
    If this condition holds, rearrange the $k_i$'s (and accordingly the $\ell_i$'s) so that $y_1 = y_{i_1}$ and $y_1 > y_{i_2} > \dots > y_{i_{n-2}}$, and denote $y_j = y_{i_j}$, $2\leq j \leq n-2$.

    \item\label{thm:Apery_KW_4} $\ell_i = y_i - y_{i+1}-1$ for all $1\leq i \leq n-2$, where $y_{n-1}:= 0$.
    
\end{enumerate}

If the properties \ref{thm:Apery_KW_1}-\ref{thm:Apery_KW_4} hold, take $q = \alpha p+a$ for some $\alpha\in \Z_{>0}$, and let $H = \la p,q,h_1,\ldots,h_{n-2} \ra$ defined by $h_i = pq-x_ip-y_iq$, for some sequence $0<x_1<x_2<\dots<x_{n-2} \leq q/2$ (e.g., $x_i = i$). Then, $\mP = \mP(H)$ is the Ap\'ery poset of $H$.
\end{theorem}

\begin{figure}[htbp]
\centering
\includegraphics[width=0.5\textwidth]{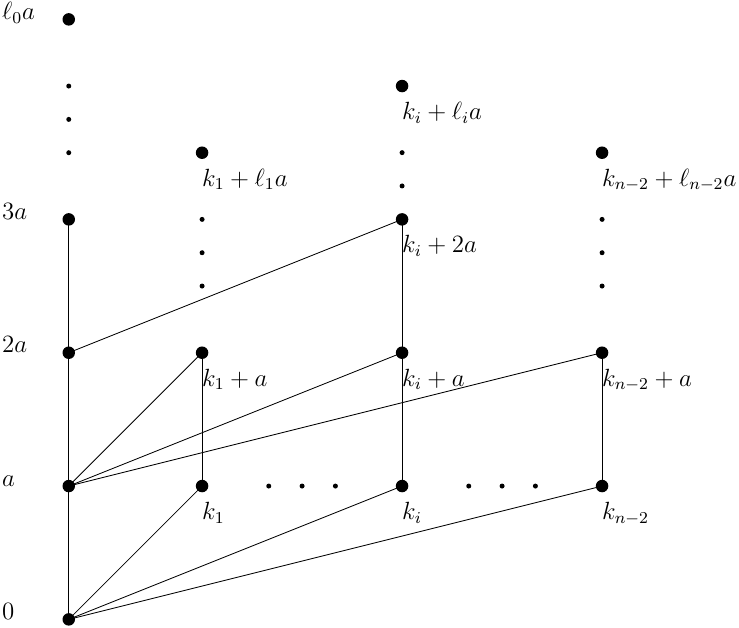}
\caption{Hasse diagram of a poset $\mP$ satisfying condition \ref{thm:Apery_KW_1} in Theorem \ref{thm:Apery_KW}.}
\label{fig:Hasse}
\end{figure}

\begin{proof}
Take $H \in KW(p,q)$ for some $q>p$ coprime with $p$. Then, $H = \la p,q,h_1,\ldots,h_{n-2} \ra$, where for all $1\leq i \leq n-2$, $h_i = pq-x_ip-y_iq$, for some $0<x_1<x_2< \dots <x_{n-2} \leq p/2$ and $q/2 \geq y_1 > y_2 > \dots > y_{n-2} > 0$. By Proposition \ref{prop:Apery_KW}, setting $y_{n-1} := 0$, one has that 
\[{\rm Ap}(H) = \{\lambda q \mid 0\leq \lambda < p-y_1 \} \cup \left( \cup_{i=1}^{n-2} \{h_i +\lambda q \mid 0\leq \lambda < y_i-y_{i+1}\} \right) .\]
We write $\ell_0 = p-y_1-1$ and $\ell_i = y_i-y_{i+1} -1$ for all $1\leq i \leq n-2$, and let $0\leq a,k_i \leq p-1$ such that $a \equiv q \pmod{p}$ and $k_i \equiv \ell_i \pmod{p}$ for all $i$. Then, it is clear that $p = (\sum_{i=0}^{n-2} \ell_i) +n-1$ and $\ell_i<\ell_0$ for all $1\leq i \leq n-2$, since $\ell_0 = p-y_1-1 \geq \lfloor p/2 \rfloor -1$ and $\ell_i \leq \lfloor p/2 \rfloor -2$. Hence, \ref{thm:Apery_KW_1} follows from Proposition \ref{prop:Apery_KW}, and \ref{thm:Apery_KW_2}-\ref{thm:Apery_KW_4} hold by the definitions above. 

Conversely, consider a poset $\mP$ satisfying properties \ref{thm:Apery_KW_1}-\ref{thm:Apery_KW_4}. Take $q = \alpha p+a$ for some $\alpha\in \Z_{>0}$, and let $H = \la p,q,h_1,\ldots,h_{n-2} \ra$ defined by $h_i = pq-x_i-y_iq$, for any sequence $0<x_1<x_2<\dots<x_{n-2} \leq q/2$. Since $\gcd(p,q) = \gcd(a,p) = 1$, one has that $H \in KW(p,q)$ and the Ap\'ery poset of $H$ is $\mP$, again by Proposition \ref{prop:Apery_KW}.
\end{proof}

By \cite[Thm. 3.10]{BGOW20}, given a numerical semigroup $S$ minimally generated by $p<a_2<\dots<a_n$ whose Ap\'ery poset is as in Theorem \ref{thm:Apery_KW}, then $S$ lies in the interior of the same face of the Kunz cone ${\mathcal C}_p$ as a KW semigroup $H = \la p,q,h_1,\ldots,h_{n-2} \ra$. Hence, by Kunz's result \cite{Kunz87}, $\k[S]$ and $\k[H]$ have the same Betti numbers, that is, $\beta_i (\k[S]) = \beta_i(\k[H]) = i {n \choose i+1}$ for $1\leq i \leq n-1$, by Theorem \ref{KWconj}. Moreover, note that the class of numerical semigroups whose Ap\'ery poset is as in Theorem \ref{thm:Apery_KW} contains semigroups that are not KW, as Example \ref{ex:KWposet_nonKW} shows.

\begin{example} \label{ex:KWposet_nonKW}
Consider the semigroup $S = \la 5,13,14,17\ra$. Then ${\rm Ap}(S) = \{0,26,17,13,14\}$, so its Ap\'ery poset is the one shown in Figure \ref{fig:example_KWposet_nonKW}. It is clear that $\mP(S)$ satisfies the conditions of Theorem \ref{thm:Apery_KW}; note that $S$ has the same Ap\'ery poset as $H = \la 5,8,9,12\ra \in KW(5,8)$. Hence, the Betti numbers of $S$ (and $H$) are $(1,6,8,3)$.

However, $S$ is not a KW semigroup. Indeed, one has that $S\notin KW(5,13)$ because $S \not\subset \la 5,13,9\ra$, $S\notin KW(5,14)$ because $S \not\subset \la 5,7\ra$, and $S\notin KW(5,17)$ because $S \not\subset \la 5,17,11\ra$.
\end{example}

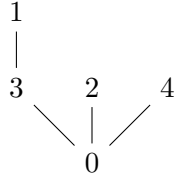
\begin{figure}[htbp]
\centering
\begin{tikzpicture}
    \node (0) at (0,0) {$0$};
    \node (16) at (-1,1) {$3$};
    \node (32) at (-1,2) {$1$};
    \node (17) at (0,1) {$2$};
    \node (14) at (1,1) {$4$};
    \draw (0) -- (16) -- (32);
    \draw (0) -- (17);
    \draw (0) -- (14);
\end{tikzpicture}
\caption{Hasse diagram of the poset $\mP(S)$ in Example \ref{ex:KWposet_nonKW}.}
\label{fig:example_KWposet_nonKW}
\end{figure}

However, having Betti numbers $\beta_i(\k[S]) = i {n\choose i+1}$, $1\leq i \leq n-1$, does not imply that the Ap\'ery poset of S satisfies the conditions of Theorem \ref{thm:Apery_KW}, as the following example shows.

\begin{example} \label{ex:Ap_arithm}
Take $S = \la 17,19,21,23,25 \ra$. Then, the defining ideal of $S$ is given by the $2\times 2$ minors of a $2\times 5$ matrix, by \cite{GSS13}, and hence the Betti numbers of $S$ are $(1,10,20,15,4)$. However, one has that $44 = 19+25 = 21+23$ in ${\rm Ap}(S)$, and hence $2,4,6,8 \precdot 10$ in ${\mathcal P}(S)$, so  it cannot be as in Theorem \ref{thm:Apery_KW}.
\end{example}

\section{Tangent cone of KW semigroups} \label{sec:tangent_cone}

Let $S$ be a numerical semigroup minimally generated by $a_1,\ldots,a_n$, with multiplicity $a_1$, i.e. $a_1<a_i$ for all $2\leq i \leq n$. For an element $s = \sum_{i=1}^n \alpha_i a_i \in S$, the tuple $\alpha = (\alpha_1,\ldots,\alpha_n) \in \N^n$ is called a {\it factorization} of $s$. Moreover, $\alpha$ is said to be a {\it maximal factorization} if $|\alpha| = \sum_{i=1}^n \alpha_i$ is maximal among all the factorizations of $s$. The semigroup $S$ is called {\it homogeneous} if for every $s\in {\rm Ap}(S)$, all the factorizations of $s$ are maximal.

Consider $I_S \subset \k[t_1,\ldots,t_n]$ the defining ideal of $S$, and denote by $G(S)$ the tangent cone of $S$ at the origin, 
\[G(S) = {\rm gr}_{\mathfrak m}(\k[S]) \cong \k[t_1,\ldots,t_n] / (I_S)_* ,\]
where $\mathfrak m$ is the homogeneous maximal ideal of $\k[S]$, $(I_S)_* = \la f_* \mid f\in I_S\ra$ is the initial ideal of $I_S$, and $f_*$ denotes the homogeneous summand of $f$ with least degree for all $f\neq 0$.

In this section, we provide a necessary and sufficient condition for $G(H)$ to be Cohen-Macaulay, when $H$ is a KW semigroup. For this purpose, first we recall a result by Jafari and Zarzuela in which they characterize the Cohen-Macaulayness of the ring $G(S)$ of a a numerical semigroup $S$ in terms of the binomials in the ideal $I_S$.

\begin{theorem}[{\cite[Thm. 3.12 and 3.17]{JZ18}}] \label{thm:JZ_3.12}
Let $S$ be a numerial semigroup minimally generated by $a_1,\ldots,a_n$ with multiplicity $a_1$. The following are equivalent:
\begin{enumerate}[(a)]
    \item \label{thm:JZ_3.12(a)} 
    $S$ is homogeneous and $G(S)$ is Cohen-Macaulay.
    
    \item \label{thm:JZ_3.12(b)} 
    For all $t^\alpha-t^\beta \in I_S$ with $|\alpha|= \sum_{i=1}^n \alpha_i > |\beta| = \sum_{i=1}^n \beta_i$, we have $\sum_{i=1}^n \alpha_i a_i \notin {\rm Ap}(S)$, and if $\bar{\beta}$ is a maximal factorization of $\sum_{i=1}^n \bar{\beta}_i a_i$, then $\alpha_1\geq \beta_1$, where 
    \[\bar{\beta} = \begin{cases}
        \beta & \text{ if } \beta_1 = 0\\
        (\beta_1-1,\beta_2,\ldots,\beta_n) & \text{ if } \beta_1 > 0.
    \end{cases}\]

    \item \label{thm:JZ_3.12(c)} 
    There exists a minimal set of binomial generators $E$ for $I_S$ such that for all $t^\alpha-t^\beta \in E$ with $|\alpha|> |\beta|$, we have $\alpha_1 \neq 0$.
\end{enumerate}
Moreover, if one of the equivalent conditions above hold, then $S$ is of homogeneous type, i.e., $\beta_i(\k[S]) = \beta_i(G(S))$ for all $i$.
\end{theorem}

Let $p<q$ relatively prime and consider a KW semigroup $H = \la p,q,h_1,\ldots,h_{n-2} \ra \in KW(p,q)$, for some $n\geq 3$. There are two sequences of integers $0<x_1<x_2<\dots <x_{n-2} \leq q/2$ and $p/2\geq y_1>y_2>\dots>y_{n-2} >0$ such that $h_i = pq-x_ip-y_iq$, $1\leq i \leq n-2$.

By \cite[App. A]{KW17}, the defining ideal of $H$, $I_H \subset \k[u,v,u_1,\ldots,u_{n-2}]$, is minimally generated by the following ${n \choose 2}$ $H$-homogeneous binomials:
\begin{equation} \label{eq:mingens}
\begin{split}
    f_{ij} &= u_iu_j-u^{q-x_i-x_j}v^{p-y_i-y_j}, \quad 1 \leq i \leq j \leq n-2 \\
    g_i &= v^{y_i-y_{i+1}}u_i-u^{x_{i+1}-x_i}u_{i+1}, \quad 1 \leq i \leq n-3 \\
    \eta_1 &= v^{p-y_1}-u^{x_1}u_1 \\
    \eta_2 &= v^{y_{n-2}}u_{n-2}-u^{q-x_{n-2}}.
\end{split}
\end{equation}

\begin{remark} \label{rem:IH_deg}
We have the following:
\begin{itemize}
    \item For all $1\leq i \leq j \leq n-2$, $(q-x_i-x_j)+(p-y_i-y_j) > 2$. This is because $q-x_i-x_j \geq 2$ and $p-y_i-y_j \geq 2$, and the inequality is strict if $q$ or $p$ is odd, respectively. Hence, $(f_{ij})_* = u_iu_j$. 
    \item $y_{n-2} +1 \leq q-x_{n-2}$. This is because if $p$ is even, then $y_{n-2}+1 \leq p/2+1\leq (q+1)/2 = q-\lfloor q/2 \rfloor \leq q-x_{n-2}$; and if $p$ is odd, then $y_{n-2}+1 \leq \lfloor p/2\rfloor+1 \leq \lfloor q/2 \rfloor \leq q-\lfloor q/2 \rfloor \leq q-x_{n-2}$. Hence, either $(\eta_2)_* = \eta_2$ if $y_{n-2}+1 = q-x_{n-2}$, or $(\eta_2)_* = v^{y_{n-2}}u_{n-2}$ otherwise.
\end{itemize}

\end{remark}

The main result of this section is Theorem \ref{thm:tangentcone_CM}, where we characterize when the tangent cone of $H$, $G(H)$, is Cohen-Macaulay. For this purpose, we use conditions \ref{thm:JZ_3.12(b)} and \ref{thm:JZ_3.12(c)} of Theorem \ref{thm:JZ_3.12}.

\begin{theorem} \label{thm:tangentcone_CM}
For a semigroup $H = \la h_1,\ldots,h_{n-2} \ra \in KW(p,q)$ of embedding dimension $n$, one has that $G(H)$ is Cohen-Macaulay if and only if 
\begin{enumerate}[(1)]
    \item\label{thm:tangentcone_CM(1)} $x_{i+1}-x_i \geq y_i-y_{i+1}$ for all $i\in \{1,\ldots,n-3\}$, and
    \item\label{thm:tangentcone_CM(2)} $x_1+y_1 \geq p-1$.
\end{enumerate}
When $G(H)$ is Cohen-Macaulay, then its Betti numbers are $\beta_0 = 1$ and $\beta_i = i {n \choose i+1}$, $1\leq i \leq n-1$.
\end{theorem}

\begin{proof}
If \ref{thm:tangentcone_CM(1)} and \ref{thm:tangentcone_CM(2)} hold, using these two conditions and Remark \ref{rem:IH_deg}, the minimal generating set of $I_H$ described in equation \eqref{eq:mingens} satisfies condition \ref{thm:JZ_3.12(c)} of Theorem \ref{thm:JZ_3.12}, so $G(H)$ is Cohen-Macaulay. Moreover, one has that $\beta_i(G(H)) = \beta_i(\k[H]) = i {n \choose i+1}$ for all $1\leq i \leq n-1$, by Theorem \ref{KWconj}.

Conversely, suppose that \ref{thm:tangentcone_CM(1)} does not hold for some $i \in \{1,\ldots,n-3\}$, i.e., $x_{i+1}-x_i < y_i - y_{i+1}$, and let us find a binomial $g$ contradicting condition \ref{thm:JZ_3.12(b)} of Theorem \ref{thm:JZ_3.12}. This will imply that $G(H)$ is not Cohen-Macaulay, by Theorem \ref{thm:JZ_3.12}, since we already know that $H$ is homogeneous (\cite[Rem. 3]{GSS25}).
We describe a finite procedure to find such a binomial $g$.

Start with $g = v^{y_i-y_{i+1}}u_i-u^{x_{i+1}-x_i}u_{i+1}$. Denote $M(g) = g_* / u = u^{x_{i+1}-x_i-1} u_{i+1}$ the tail of $g$ divided by $u$ and let $s(g) = (x_{i+1}-x_i-1)p+h_{i+1} \in H$ be its $H$-degree. 

\begin{enumerate}[label*=\arabic*.]
    \item If $M(g)$ corresponds to a maximal factorization of $s(g)$, then 
    we are done.
    \item Otherwise, let $M' \in \k[u,v,u_1,\ldots,u_{n-2}]$ be the monomial corresponding to a maximal factorization of $s(g)$.

    \begin{enumerate}[label*=\arabic*.]
        \item If $u$ does not divide $M'$, then $u$ must divide $M(g)$. Otherwise, one would have that the $H$-degree of $M(g)$ is $h_{i+1} \in {\rm Ap}(H)$, and hence $H$ would not be homogeneous. Take $\tilde{g} = M'-M(g)$ and start again with $\tilde{g}$.
        
        \item If $u$ divides $M'$, let $\ell \geq 1$ be maximal such that $u^\ell $ divides $M'$. Since $h_{i+1}$ is a minimal generator of $H$, one must have $\ell < x_{i+1}-x_i-1$, so $u^\ell$ also divides $M(g)$. Set $\tilde{g} = (M'-M(g)) / u^\ell$ and start again with $\tilde{g}$.
    \end{enumerate}

    \noindent Note that every time we replace the binomial $g$ by a new binomial $\tilde{g}$, the binomial $\tilde{g} \in I_H$ remains non homogeneous, while the exponent of $u$ in the tail monomial strictly decreases. Therefore, the procedure terminates after finitely many steps.
\end{enumerate}

If \ref{thm:tangentcone_CM(2)} does not hold, one can use a similar construction starting from the binomial $g = v^{p-y_1}-u^{x_1}u_1$ to conclude that $G(H)$ is not Cohen-Macaulay in this case. This ends the proof.
\end{proof}

\begin{remark}
\begin{enumerate}
    \item Geometrically, condition \ref{thm:tangentcone_CM(1)} in Theorem \ref{thm:tangentcone_CM} says that the steps in the lattice path defining $H$ are wider than tall; see Figure \ref{fig:lattice_path_tangent_cone}.

\begin{figure}[htbp]
\begin{subfigure}[b]{0.45\textwidth}
\centering
\includegraphics[width=0.8\textwidth]{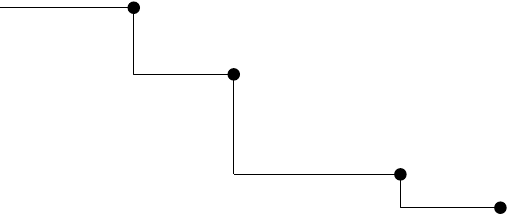}
\end{subfigure}
\begin{subfigure}[b]{0.45\textwidth}
\centering
\includegraphics[width=0.5\textwidth]{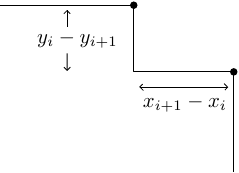}
\end{subfigure}
\caption{Lattice path satisfying condition \ref{thm:tangentcone_CM(1)} in Theorem \ref{thm:tangentcone_CM} (left), and a zoomed-in view of one of its steps (right).}
\label{fig:lattice_path_tangent_cone}
\end{figure}

    \item If $H = \la p,q,h_1,\ldots,h_{n-2}\ra  \in KW_D(p,q)$, i.e., if it is a KW semigroup whose defining ideal is generated by the $2\times 2$ minors of a $2\times n$ matrix, by \cite[Thm. 1.4]{GSS25} there exist $x,y\in \Z_{>0}$ such that $(n-2)x \leq q/2$ and $(n-2)y \leq p/2$, and $h_i = pq-ixp-(n-1-y)iq$ for $1\leq i \leq n-2$. In this case, $G(H)$ is Cohen-Macaulay if and only if $x>y$ and $x+(n-2)y \geq p-1$.
\end{enumerate}

\end{remark}

As a direct consequence, we get the minimal number of generators of $(I_H)_*$ when $G(H)$ is Cohen-Macaulay, and also a minimal generating set of this ideal.

\begin{corollary}
Let $H = \la p,q,h_1,\ldots,h_{n-2} \ra  \in KW(p,q)$ such that 
\begin{enumerate}[(1)]
    \item $x_{i+1}-x_i \geq y_i-y_{i+1}$ for all $i\in \{1,\ldots,n-3\}$, and
    \item $x_1+y_1 \geq p-1$.
\end{enumerate}
Then, the minimal number of generators of $(I_H)_*$ is $\mu((I_H)_*) = {n \choose 2}$. Moreover, $(I_H)_*$ is minimally generated by the ${n \choose 2}$ monomials and binomials:
\[\begin{array}{rcll}
    (f_{ij})_* & = &  u_iu_j, & \quad 1 \leq i \leq j \leq n-2, \\
    (g_i)_* & = & v^{y_i-y_{i+1}}u_i-\epsilon_i u^{x_{i+1}-x_i}u_{i+1}, & \quad 1 \leq i \leq n-3, \\
    (\eta_1)_* & = & v^{p-y_1}- \epsilon_1' u^{x_1}u_1, & \\
    (\eta_2)_* & = & v^{y_{n-2}}u_{n-2} - \epsilon_{n-2}' u^{q-x_{n-2}}, &
\end{array}\]
where $\epsilon_i = 1$ if $x_{i+1}-x_i = y_i-y_{i+1}$ and $0$ otherwise; $\epsilon_1' = 1$ if $x_1+y_1=p-1$ and $0$ otherwise; and $\epsilon_{n-2}' = 1$ if $y_{n-2}+1=q-x_{n-2}$ and $0$ otherwise.
\end{corollary}

\begin{proof}
By Theorem \ref{thm:tangentcone_CM}, $\mu((I_H)_*) = \mu(I_H) = {n \choose 2}$ in this case. Moreover, one has that \[J:= \la \{(f_{ij})_* \mid 1\leq i\leq j \leq n-2 \} \cup \{(g_i)_* \mid 1\leq i \leq n-3 \} \cup \{(\eta_1)_*,(\eta_2)_*\} \ra \subset (I_H)_*.\] 
One can check that all the generators of $J$ are minimal, so $\mu(J) = {n \choose 2}$ and hence $J = (I_H)_*$.
\end{proof}

The next example shows that this $\mu((I_H)_*)$ may be bigger than ${n \choose 2}$ when $G(H)$ is not Cohen-Macaulay, and this number does not only depend on the embedding dimension of $H$.

\begin{example}

\begin{enumerate}
\item Let $H_1 = \la 8,17,60,69,78\ra \in KW(8,17)$. $H_1$ is defined by the sequences $x_1=1<x_2=2<x_3=3$ and $y_1=4>y_2=3>y_3=2$. Note that $G(H_1)$ is not Cohen-Macaulay since $x_1+y_1 = 5 < 7$. A direct computation with Singular shows that the Betti numbers of $G(H_1)$ are $(1,    11,    24,    21,     8,     1)$, and hence, $\mu({I_{H_1}}_*) = 11 > {5 \choose 2}$.

\item Let $H_2 = \la 8,17,69,70,71\ra \in KW(8,17)$. $H_2$ is defined by the sequences $x_1=2<x_2=4<x_3=6$ and $y_1=3>y_2=2>y_3=1$. Note that $G(H_2)$ is not Cohen-Macaulay since $x_1+y_1 = 5 < 7$. A direct computation with Singular shows that the Betti numbers of $G(H_2)$ are $(1,    12,    28,    27,    12,     2)$, and hence, $\mu({I_{H_2}}_*) = 12 > {5 \choose 2}$.

\item Let $H_3 = \la 8,17,44,70,63\ra \in KW(8,17)$. $H_3$ is defined by the sequences $x_1=3<x_2=4<x_3=7$ and $y_1=4>y_2=2>y_3=1$. Note that $G(H_3)$ is not Cohen-Macaulay since $x_2-x_1 < y_1-y_2$. A direct computation with Singular shows that the Betti numbers of $G(H_3)$ are $(1,    11,    24,    21,     8,     1)$, and hence, $\mu({I_{H_3}}_*) = 11 > {5 \choose 2}$.
\end{enumerate}

\end{example}


\begin{thebibliography}{99}

\bibitem{BGOW20}{
W. Bruns, P. Garc\'ia-S\'anchez, C. O’Neill, and D. Wilburne,
Wilf’s conjecture in fixed multiplicity,
{\em Internat. J. Algebra Comput.} {\bf 30} (2020), 861--882.
}

\bibitem{GSS13}{
P. Gimenez, I. Sengupta, and H. Srinivasan,
Minimal graded free resolutions for monomial curves defined by arithmetic sequences,
{\em J. Algebra} {\bf 388} (2013), 294--310.
}

\bibitem{GS19}{
P. Gimenez and H. Srinivasan,
The structure of the minimal free resolution of semigroup rings obtained by gluing,
{\em J. Pure Appl. Algebra} {\bf 223} (2019), 1411--1426.
}

\bibitem{GSS25}{
M. Gonz{\'a}lez-S{\'a}nchez, S. Singh, and H. Srinivasan,
The Betti numbers of Kunz-Waldi semigroups,
{\em Proc. Amer. Math. Soc.} {\bf 153} (2025), 4215--4224.
}

\bibitem{Kunz87}{
E. Kunz,
\"Uber die Klassifikation numerischer Halbgruppen,
{\em Fakult{\"a}t Mathematik der Universit{\"a}t Regensburg}, vol. 11 (1987).
}

\bibitem{KW14}{
E. Kunz and R. Waldi, 
Geometrical illustration of numerical semigroups and of some of their invariants,
{\em Semigroup Forum} {\bf 89} (2014), 664--691.
}

\bibitem{KW17}{
E. Kunz and R. Waldi, 
On the deviation and the type of certain local Cohen–Macaulay rings and numerical semigroup,
{\em J. Algebra} {\bf 478} (2017), 397--409.
}

\bibitem{JZ18}{
R. Jafari and S. Zarzuela,
Homogeneous numerical semigroups,
{\em Semigroup Forum} {\bf 97} (2018), 278--306.
}

\bibitem{Mor87}{
M. Morales,
Syzygies of monomial curves and a linear diophantine problem of Frobenius,
{\em MPIM Preprint Series} {\bf 48} (1987), 17 pp.
}

\bibitem{Mor91}{
M. Morales,
Noetherian Symbolic Blow-Ups,
{\em J. Algebra} {\bf 140} (1991), 12--25.
}


\bibitem{SS25}{
S. Singh and H. Srinivasan, 
A Class of Numerical Semigroups Defined by Kunz and Waldi - Their Principal Matrices and Structure,
{\em J. Algebra. Appl.} {\bf 24} (2025), 20 pp.
}

\end{thebibliography}
\end{document}